\documentclass[12pt,english]{amsart}
\usepackage[paperwidth=210mm,paperheight=297mm,inner=3cm,outer=3cm,top=2.5cm,bottom=2.6cm]{geometry}

\usepackage[T1]{fontenc}
\usepackage{amssymb}

\usepackage{paralist}
\usepackage{babel}
\usepackage[pdftex]{graphicx}    
\usepackage[leqno]{amsmath}
\usepackage{amsthm}
\usepackage{url,  mathrsfs}
\usepackage{hyperref}

\newcommand\mbb{\mathbb}

\newcommand\mcal{\mathcal}

\newcommand\ol{\overline}

\newcommand\sV{\mcal{V}}

\newcommand\C{\mbb{C}}

\renewcommand\P{\mbb{P}}
\newcommand\Q{\mbb{Q}}
\newcommand\R{\mbb{R}}

\renewcommand\epsilon{\varepsilon}

\renewcommand\phi{\varphi}

\renewcommand\theta{\vartheta}

\theoremstyle{plain}
\newtheorem{Thm}{Theorem}

\newtheorem*{Thm*}{Theorem}
\newtheorem*{Prop*}{Proposition}
\newtheorem*{Cor*}{Corollary}
\newtheorem*{Lemma*}{Lemma}
\newtheorem*{Sublemma*}{Sublemma}
\newtheorem*{Conjecture*}{Conjecture}

\theoremstyle{definition}

\newtheorem{Example}[Thm]{Example}

\newtheorem{Remark}[Thm]{Remark}

\newtheorem*{Constr*}{Construction}
\newtheorem*{Def*}{Definition}
\newtheorem*{Defs*}{Definitions}
\newtheorem*{Example*}{Example}
\newtheorem*{Examples*}{Examples}
\newtheorem*{Exercise*}{Exercise}
\newtheorem*{LemmaDef*}{Lemma and Definition}
\newtheorem*{Notation*}{Notation}
\newtheorem*{Problem*}{Problem}
\newtheorem*{Question*}{Question}
\newtheorem*{Remark*}{Remark}
\newtheorem*{Remarks*}{Remarks}
\newtheorem*{Warning*}{Warning}

\newtheorem*{Text*}{}

\begin{document}

\title[Computing Hermitian Determinantal Representations]{Computing Hermitian Determinantal Representations of Hyperbolic
  Curves}

\begin{abstract}
Every real hyperbolic form in three variables can be realized as the determinant of a linear net of Hermitian matrices containing a positive definite matrix. Such representations are an algebraic certificate for the hyperbolicity of the polynomial and their existence has been proved in several different ways. However, the resulting algorithms for computing determinantal representations are computationally intensive. In this note, we present an algorithm that reduces a large part of the problem to linear algebra and discuss its numerical implementation.
\end{abstract}

\author{Daniel Plaumann}
\address{Universit\"at Konstanz, Germany} 
\email{Daniel.Plaumann@uni-konstanz.de}
\author{Rainer Sinn}
\address{Georgia Institute of Technology, Atlanta, GA, USA} 
\email{rsinn3@math.gatech.edu}
\author{David E Speyer}
\address{University of Michigan, Ann Arbor, MI, USA} 
\email{speyer@umich.edu}
\author{Cynthia Vinzant}
\address{North Carolina State University, Raleigh, NC, USA} 
\email{clvinzan@ncsu.edu}
\maketitle

\section*{Introduction}

Let $f$ be a real homogeneous polynomial of degree $d$ in 
variables $x,y,z$. A \textbf{Hermitian determinantal
  representation} of $f$ is an expression
\begin{equation}\label{eq:intro}
 f \;\;=\;\;\det(xM_1+yM_2+zM_3),
\end{equation}
where $M_1,M_2,M_3$ are Hermitian $d\times d$ matrices. The representation is \textbf{definite} if
there is a point $e\in\R^{3}$ for which the matrix $e_1M_1+e_2M_2+e_3M_3$ is positive
definite.

This imposes an immediate condition on the projective curve $\sV_\C(f)$.
Because the eigenvalues of a Hermitian matrix are real, 
every real line passing through $e$ meets
this hypersurface in only real points.  A~polynomial with this
property is called \textbf{hyperbolic} (with respect to $e$). 
Hyperbolicity is reflected in the topology of the real points $\sV_\R(f)$.
When the curve $\sV_{\C}(f)$ is smooth, $f$ is hyperbolic if and only if $\sV_\R(f)$ 
consists of $\lfloor\frac{d}{2}\rfloor$ nested
ovals, and a pseudo-line if $d$ is odd. 

\begin{figure}[h]
 \includegraphics[width=3.8cm]{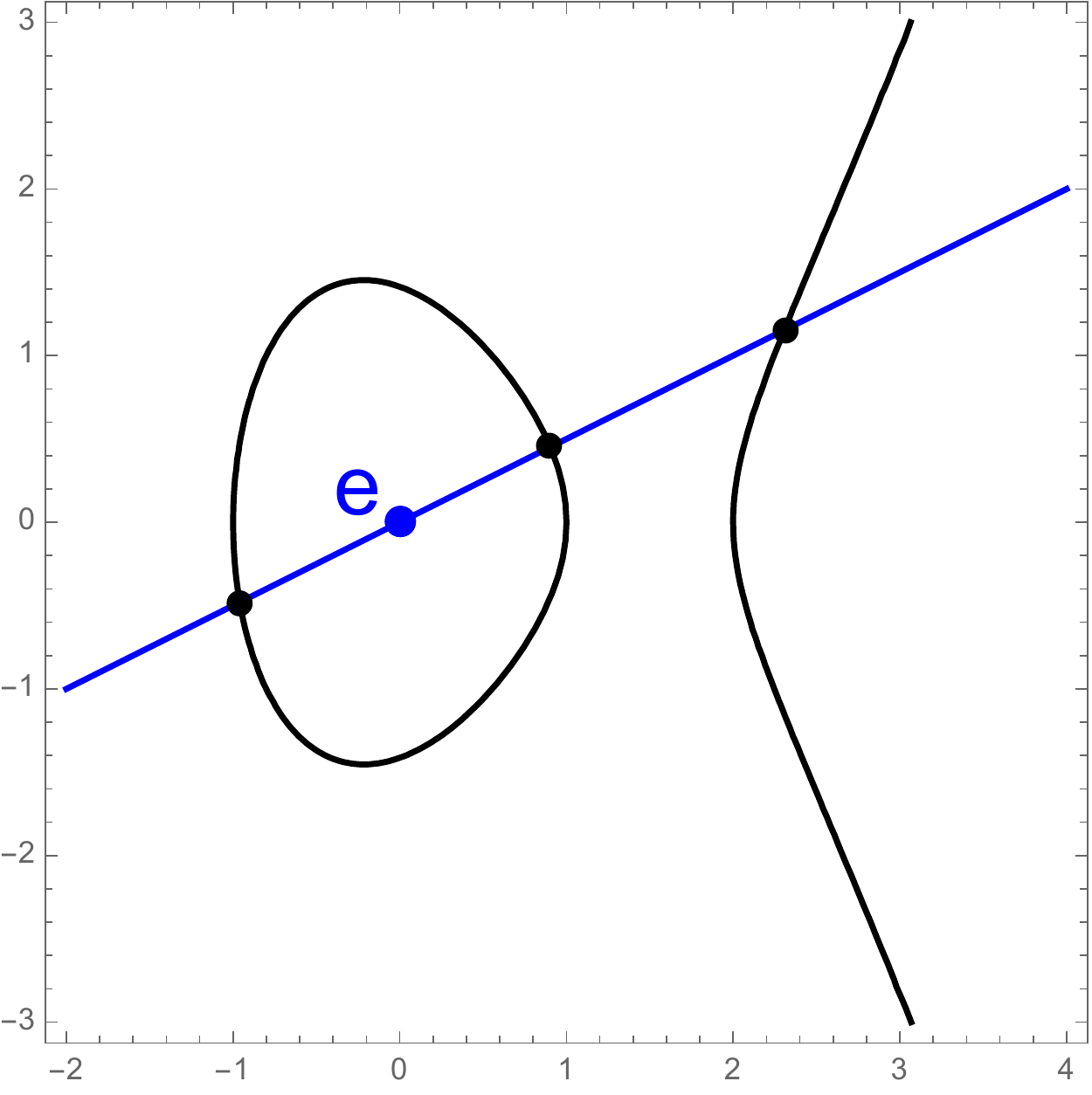} \quad \quad \quad
  \includegraphics[width=4.0cm]{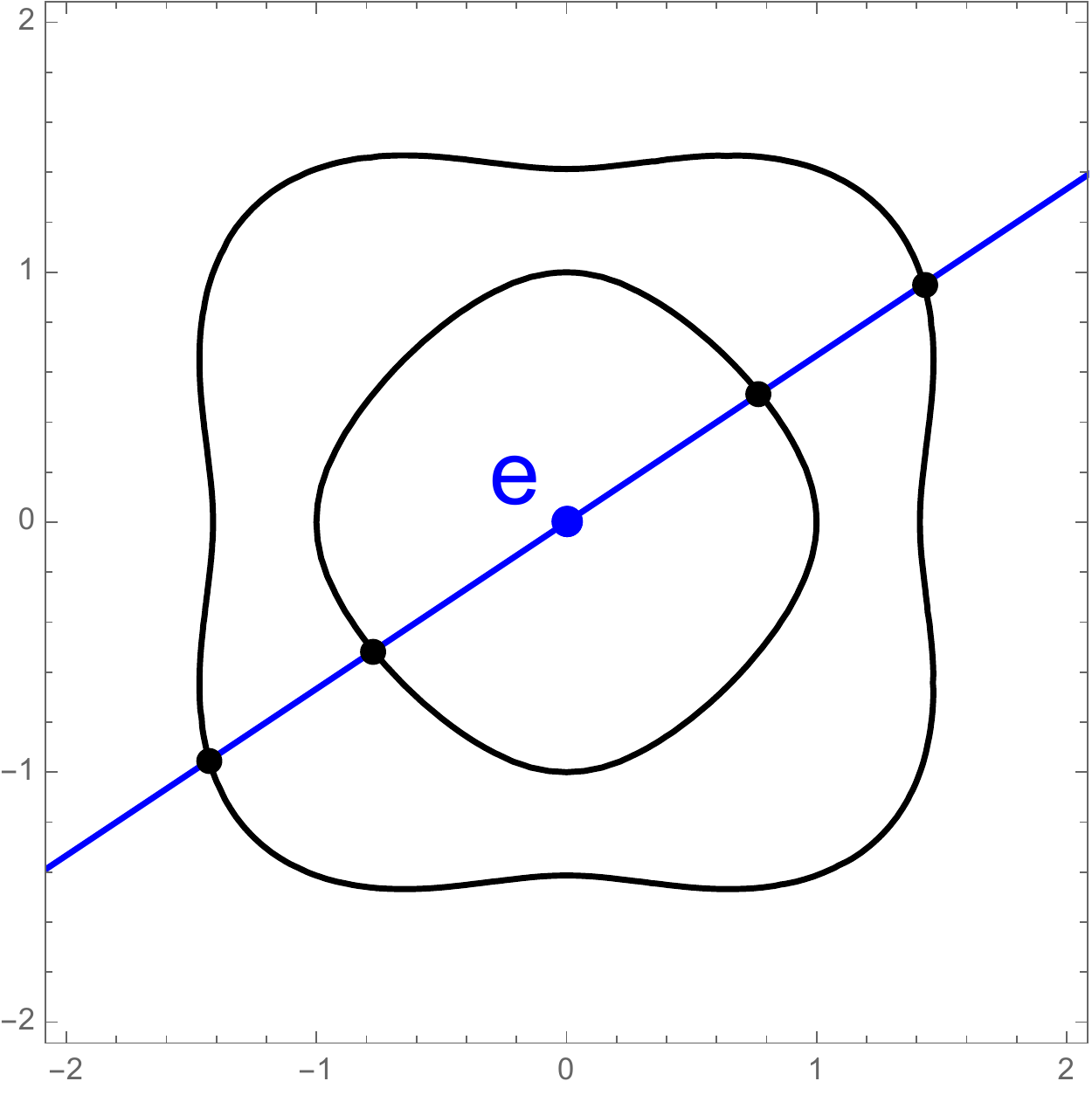}
\caption{Cubic and quartic hyperbolic curves in $\P^2(\R)$.  }
\label{fig:cubic}
\end{figure}

The Helton-Vinnikov theorem \cite{MR2292953}
(previously known as the Lax conjecture) 
says that every
hyperbolic polynomial in three variables possesses a definite determinantal
representation \eqref{eq:intro} with real symmetric matrices. 
Thus given a hyperbolic plane curve, one can investigate the problem of 
computing a definite determinantal representation. 

Computing symmetric determinantal representations 
of hyperbolic plane curves both symbolically and numerically 
was investigated by Sturmfels and two of the current authors
in  \cite{MR2962788} and in the case of quartic curves in \cite{MR2781949}.
Recently, it was discovered \cite{MR3066450} that looking for Hermitian matrices 
$M_1, M_2, M_3$, rather than 
real symmetric matrices, greatly simplifies this computational problem and the
proof of Helton and Vinnikov's theorem.

The goal of this paper is to present an algorithm for computing 
determinantal representations \eqref{eq:intro}, examine this 
algorithm numerically, and to compare it with existing methods. 
This construction is based heavily on  \cite{MR3066450}, which 
generalizes a classical construction due to Dixon \cite{Dixon02}. 

\emph{Acknowledgements.} We would like to thank Arno Fehm for helpful discussions concerning Hilbert's Irreducibility Theorem.
Cynthia Vinzant was partially supported by an NSF grant DMS-1204447.

\section{The Algorithm}

The input to our algorithm is a polynomial $f\in \R[x,y,z]$ of degree $d$
with smooth complex variety $\sV_{\C}(f)$ and a point $e=(e_1,e_2,e_3)\in \R^3$ with respect to which
$f$ is hyperbolic.  We will intersect the curve $\sV_{\C}(f)$ with the degree $(d-1)$ curve given by 
the directional derivative
\begin{equation}\label{eq:der}
 g(x,y,z) \;\;= \;\; e_1\frac{\partial f}{\partial x}+e_2\frac{\partial f}{\partial y}+e_3\frac{\partial f}{\partial z}. 
 \end{equation}
We assume that the intersection $\sV_{\C}(f)\cap \sV_{\C}(g)$ in $\P^{2}$ is transverse.
That is, the two curves $\sV_{\C}(f)$ and $\sV_{\C}(g)$ intersect in $d(d-1)$ distinct points. 
In fact, this implies none of these intersection points are real \cite[Lemma 2.4]{MR3066450}. 
If the intersection $\sV_{\C}(f)\cap \sV_{\C}(g)$ is not transverse, we may replace $g$ with the 
directional derivative of $f$ in direction $e'$, taking $e'$ to be a generic point sufficiently close to $e$. 

The output of the algorithm will be three Hermitian $d\times d$ matrices $M_1, M_2, M_3$ 
such that $f = c\cdot \det(xM_1+yM_2+zM_3)$ where $c\in \R_{>0}$ and $e_1M_1+e_2M_2+e_3M_3$ is positive definite. 
Furthermore, $g$ will be one of the diagonal minors of the resulting matrix $M = xM_1+yM_2+zM_3$, 
namely the minor of $M$ obtained by removing the first row and first column from $M$.

The construction below is based on the idea that if the Hermitian matrix $M$
is a determinantal representation of $f = \det(M)$, then its adjugate matrix $M^{\rm adj}$ 
satisfies 
\[  M^{\rm adj}\cdot M \;\;=\;\; \det(M)\cdot I \;\;=\;\ f\cdot I.\]
Let $a$ denote the top row of $M^{\rm adj}$. Then, taking the top row of this matrix equation,
we obtain the relation $a (x M_1 + y M_2 + z M_3) = (f,0,\ldots,0)$. Similar arguments give 
$(x M_1 + y M_2 + z M_3) \ol{a}^T =    (f,0,\ldots,0)^T$. We introduce a suitable 
vector $a = (a_{11}, a_{12}, \hdots, a_{1d})$  and 
solve these linear equations in the entries of the 
$M_i$. This finds $(M_1, M_2, M_3)$ without ever explicitly computing $M^{\rm adj}$. 

The algorithm proceeds as follows:\\

\begin{enumerate}
\item[(A1)] Compute the $d(d-1)$ points $\sV_{\C}(f)\cap \sV_{\C}(g)$. \smallskip
\item[(A2)]  Split the points into two disjoint, conjugate sets $\sV_{\C}(f,g) = S \cup \overline{S}$.  \smallskip
\item[(A3)]  Let $a_{11}$ equal $g$. \smallskip
\item[(A4)]  Extend $a_{11}$ to a basis $a =(a_{11},\ldots,a_{1d})$ of
  the vector space of polynomials in $\C[x,y,z]_{d-1}$ that vanish on
  the points $S$. (In particular, this space is always of dimension
  $d$; see Remark \ref{Remark:NondegenerateIntersection} below.)\smallskip
\item[(A5)]  In the $3 d^2$ variables $(M_1)_{i,j}, (M_2)_{i,j}, (M_3)_{i,j}$, solve the 
$2d\binom{d+2}{2} = (d+2)(d+1)d$ affine linear equations coming from the polynomial vector equations
\begin{eqnarray*}
a \;(xM_1+yM_2+zM_3) &= \;\;(f,0\ldots 0)\;\;\;\\
(xM_1+yM_2+zM_3)\;\ol{a}^T &= \;\;(f,0 \ldots 0)^T. \notag
\end{eqnarray*}  
\item[(A6)] Output the \textit{unique} solution $M_1, M_2, M_3$. \bigskip
\end{enumerate}
We need to argue that such a solution $M = x M_1+ yM_2+zM_3$ exists, is unique, and has the 
desired properties, which we do below.  Numerical implementation of this algorithm and surrounding 
computational issues will be discussed in Section~\ref{sec:Implementation}.

\begin{Thm}\label{thm:Alg}
Let $f\in \R[x,y,z]$ be hyperbolic with respect to a point $e\in \R^3$ with $f(e)>0$. 
Suppose that $\sV_{\C}(f)$ is smooth and that all the intersection points of $\sV_{\C}(f)$ and
$\sV_{\C}(g)$ are transverse, where 
$g=e_1\frac{\partial f}{\partial x}+e_2\frac{\partial f}{\partial y}+e_3\frac{\partial f}{\partial z}$.  Then the system of equations 
in {\rm (A5)} has a unique solution $M_1, M_2, M_3$, which are Hermitian matrices satisfying
\[ f \; = \;c\cdot\det(xM_1+yM_2+zM_3) \;\;\;\; \text{ and } \;\;\;\; e_1M_1+e_2M_2+e_3M_3 \succ 0,\]
where $c\in \R_{>0}$.
\end{Thm}

\begin{proof}
(\textit{Existence.}) First, let us show that the affine linear equations (A5) have some solution $M_1, M_2, M_3$. 
By Construction~4.5 and Theorem~4.6 of \cite{MR3066450}, 
there exists a Hermitian linear matrix  $M'= xM'_1+yM'_2+zM'_3$ such that
for some $c\neq 0$, the determinant $\det(M')$ equals $c^{d-1}f$,  
the matrix $M'(e) = e_1M'_1+e_2M'_2+e_3M'_3$ is either positive or negative definite, 
and the first row of the adjugate matrix $A=(1/c^{d-2})(M')^{\rm adj}$ 
is precisely $a =(a_{11}, a_{12}, \hdots, a_{1d})$.  
Since the matrices $M'$ and $(M')^{\rm adj}$ are Hermitian, it follows that the 
first column of $(1/c^{d-2})(M')^{\rm adj}$ is $\ol{a}^T = (\ol{a_{11}}, \ol{a_{12}}, \hdots, \ol{a_{1d}})^T$.

In fact, the constant $c$ must be positive. We can see this from examining our matrices at the point $e$. 
Since $M'(e)$ is definite, both $(M')^{\rm adj}(e)$ and $A(e)$ must be definite as well. 
Furthermore, because the $(1,1)$ entry of $A(e)$, namely $a_{11}(e)=g(e)$, is positive
we see that the matrix $A(e)$ is positive definite.  Then the equation
\[\det(A) \;=\; (1/c^{d-2})^d \det((M')^{\rm adj}) = (1/c^{d-2})^d \cdot (c^{d-1} f)^{d-1}\;=c\cdot f^{d-1}\]
evaluated in the point $e$ shows that $c$ is positive. 
To find a solution to the equations (A5), let $M=(1/c)M'$. Then \[\det(M)\;=\;(1/c)^d\det(M') \;=\; (1/c)f.\]
Furthermore $M^{\rm adj}$ equals $(1/c)^{d-1} (M')^{\rm adj}$, which is $(1/c)\cdot A$.
We know that  both $M^{\rm adj}\cdot M$ and $M\cdot M^{\rm adj}$ equal $\det(M)I$.
Dividing these identities by $(1/c)$ we see that 
$A\cdot M = f\cdot I $ and $M\cdot A =   f\cdot I.$
 From taking the first row of the first equation and the first 
column of the second equation, we see that $M$ 
satisfies the equations  
\[a \;M =(f,0\ldots 0) \;\;\;\; \text{ and } \;\;\;\; M\;\ol{a}^T = (f,0 \ldots 0)^T. \]
Since $c\cdot f = \det(M)$ and $M(e)$ is positive definite, in order to finish 
the proof, it suffices to show that this is the unique solution to these equations. \\

(\textit{Uniqueness}.) 
First we argue that if the system of equations (A5) has two distinct solutions, then 
it has two distinct solutions $M=xM_1+yM_2+zM_3$ and $N=xN_1+yN_2+zN_3$ for which both 
$\det(M)$ and $\det(N)$ are not identically zero. 

By the proof above, there is a solution $M$ to the equations (A5)
with $\det(M) \neq 0$.   If there is another solution $N$, then for any $\lambda\in \R$, the matrix 
$\lambda M + (1-\lambda)N$ is also a solution.  The set of  
 $\lambda\in \R$ for which the determinant $\det(\lambda M + (1-\lambda)N) $ is identically zero 
 is closed and does not contain $\lambda =1$. Thus for almost 
 all choices of  $\lambda \in \R$, $\det(\lambda M + (1-\lambda)N) $ is nonzero. 

Now suppose $M=xM_1+yM_2+zM_3$ and $N=xN_1+yN_2+zN_3$ are matrices satisfying the
equations (A5) for which $\det(M)$ and $\det(N)$ are not identically zero. 
We see that at a general point point $(x,y,z)$ in $\sV_{\C}(f)$, the vector $a$ is non-zero, and 
the matrix $M$ does not have full rank. Since $\det(M)$ has degree $d$ and $f$ is irreducible, 
we can conclude that $\det(M)=\alpha f$ for some constant $\alpha\neq 0$. Similarly $\det(N)=\beta f$ for
some $\beta\neq 0$. 

Again we use the identity $(M)^{\rm adj}\cdot M =  M \cdot (M)^{\rm adj}= \det(M)I = \alpha f I$.
Therefore $(M^{\rm adj})_1 \cdot M = \alpha (f,0\ldots0) = \alpha (a \cdot M)$, where $(M^{\rm adj})_1$ is the first row of $M^{\rm adj}$.
For generic $(x,y,z)$ in $\C^3$, the matrix $M$ is invertible, which implies that
 the first row of $M^{\rm adj}$ is $\alpha a$.
Similarly, the first column of $M^{\rm adj}$ is $\alpha \ol{a}^T$, and 
the first row and column of $N^{\rm adj}$ are $\beta a$ and $\beta \ol{a}^T$, respectively. 
 
The matrices $(1/\alpha)M^{\rm adj}$ and $(1/\beta)N^{\rm adj}$ therefore have the same 
first row and column and both have rank-one along the curve $\sV_{\C}(f)$. 
For generic $(x,y,z)\in \sV_{\C}(f)$, the entries of these rows and columns are non-zero. 
Therefore the difference $(1/\alpha)M^{\rm adj}- (1/\beta)N^{\rm adj}$ vanishes on $\sV_{\C}(f)$
and all of the entries of this difference must be divisible by $f$. 
However, the entries of these matrices have degree $d-1$ whereas $f$ has degree $d$, so we see that 
$(1/\alpha)M^{\rm adj}$ must equal $(1/\beta)N^{\rm adj}$.  In particular, $M^{\rm adj}$ is a constant multiple of 
$N^{\rm adj}$.  It follows that $M$ is a constant multiple of $N$. 
Because our affine linear equations (A5) are not homogenous, we see that in fact $M=N$ and our solution is unique. 
\end{proof}

\begin{Remark} In fact, in our algorithm we can replace $g$ in \eqref{eq:der} 
by any polynomial $g\in \R[x,y,z]_{d-1}$ where $g(e)>0$ and $g$ \textbf{interlaces $f$
with respect to $e$}.  By this, we mean that for every point $p\in \R^3$ the roots 
of the univariate polynomial $g(te+p)$ interlace those of $f(te+p)$.
\end{Remark}

\begin{Remark} 
When the curve $\sV(f)$ is singular, some of the above construction goes through, but 
requires more care. Because $f$ has a definite determinantal representation $f=\det(M)$, 
e.g. \cite[Cor. 4.9]{MR3066450}, 
we know that there exists an interlacer $g = (M^{\rm adj})_{11}$ 
and a basis of polynomials $a = ( (M^{\rm adj})_{11}, \hdots, (M^{\rm adj})_{1d})$ for which 
the equations (A5) have a solution. Finding such polynomials remains a challenge. When $f$ is singular, 
the intersection $\sV(f)\cap \sV(g)$ contains points with multiplicities. It is unclear exactly 
how to split these points with multiplicity and find the correct linear space of polynomials 
$a_{11}, \hdots, a_{1d}$ that vanish on them. One example of this procedure is given below. 
\end{Remark}

\begin{Example} \label{ex:quartic}
To illustrate our algorithm, we apply it to the quartic 
\begin{equation}\label{eq:exQuartic}
f(x,y,z) \;\;=\;\; x^4 - 4x^2y^2 + y^4 - 4x^2z^2 - 2y^2z^2 + z^4,
\end{equation}
which is hyperbolic with respect to the point $e=[1:0:0]$, and appears as Example~4.12 in \cite{MR3066450}. 
This curve has two nodes, $[0:1:1]$ and $[0:-1:1]$,
but done carefully, our algorithm still works. 
Figure~\ref{fig:ex} shows the real curves $\sV_{\R}(f)$ and $\sV_{\R}(g)$, where $g = \frac{\partial f}{\partial x}$, in different affine planes.

 \begin{figure}[h] 
 \includegraphics[width=1.8in]{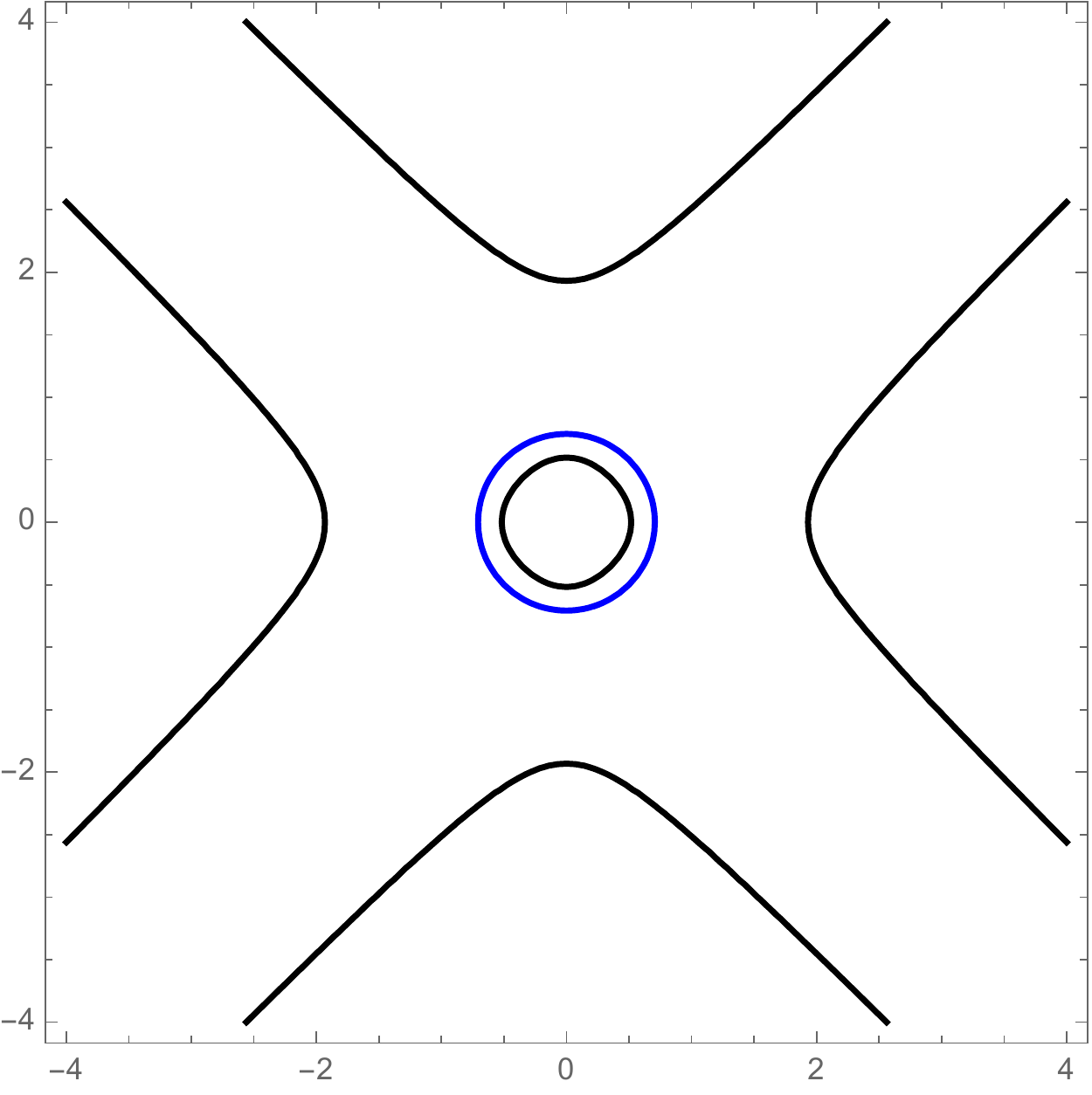} \quad \quad \quad
  \includegraphics[width=1.8in]{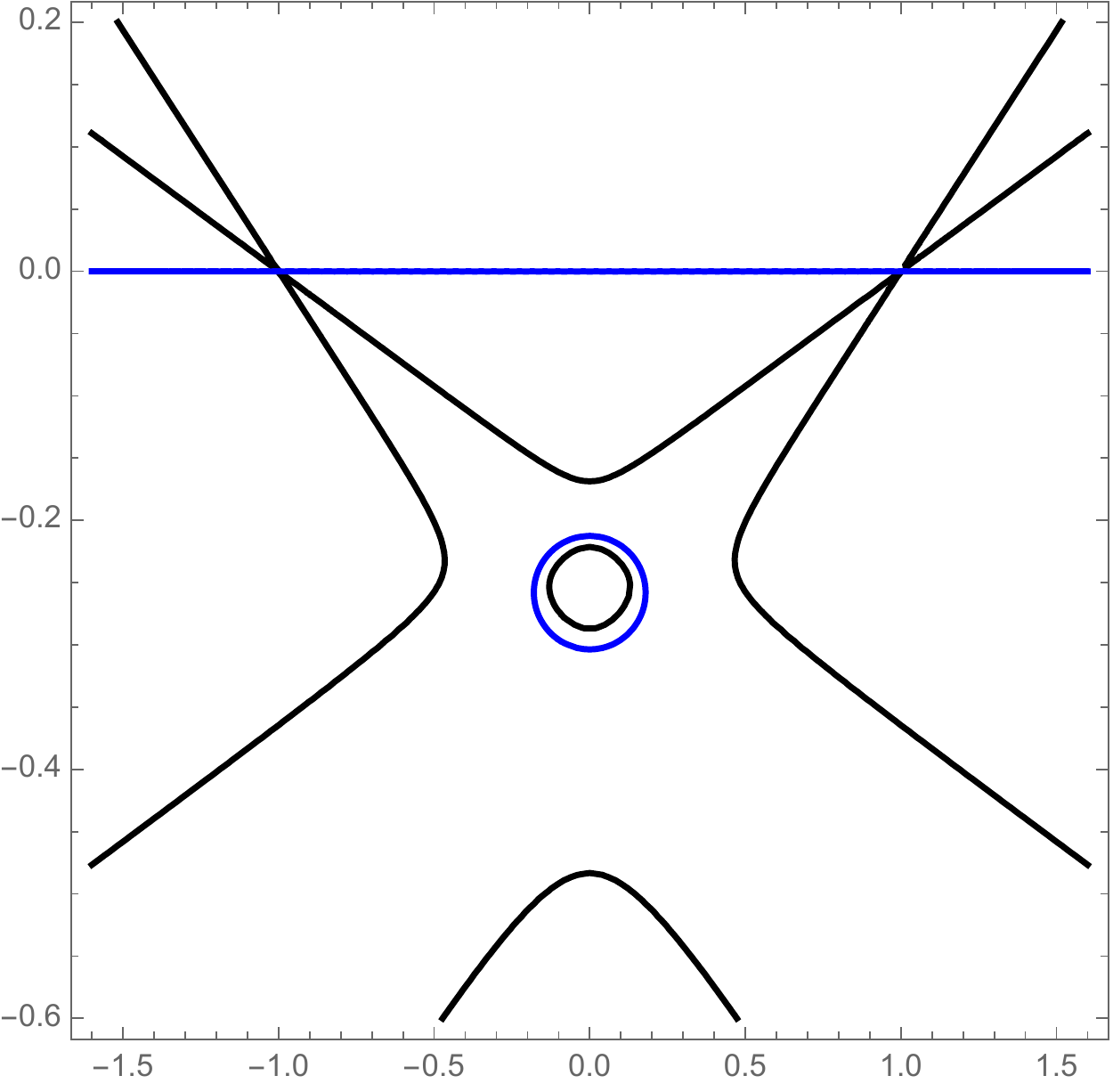} 
\caption{The hyperbolic quartic \eqref{eq:exQuartic} and its directional derivative. }\label{fig:ex}
\end{figure}

First we define $a_{11}$ to be the directional derivative 
$\frac{1}{4} D_ef = x^3 - 2 x y^2 - 2 x z^2$. 
The intersection of $f$ and $a_{11}$ consists of the eight points 
$[2:\pm\sqrt{3}:\pm i]$, $[2:\pm i:\pm\sqrt{3}]$ and the two nodes, $[0:\pm1:1]$, 
each with multiplicity 2.  We divide 
these points into two conjugate sets $S\cup\ol{S}$ where
\[ S  = \left\{ [0:1:1], \;[0:-1:1], \;[2:\sqrt{3}:i],\;[2:-\sqrt{3}:i],\;[2:i:\sqrt{3}],\;[2:i:-\sqrt{3}]\right\}.\]
The vector space of cubics in $\C[x,y,z]$ vanishing on these six points 
is four dimensional and we extend $a_{11}$ to a basis 
$\{a_{11},a_{12},a_{13}, a_{14}\}$ for this space, where\
\begin{align*}
a_{12} &= i x^3 + 4 i x y^2 - 4 x^2 z - 4 y^2 z + 4 z^3,\\ 
a_{13} &= -3 i x^3 + 4 x^2 y +  4 i x y^2 - 4 y^3 + 4 y z^2,\\
a_{14} &=-x^3 - 2 i x^2 y - 2 i x^2 z +  4 x y z.
\end{align*}
Let $M=xA+yB+zC$. The two $4\times 4$ polynomial matrix equations 
$aM = (f,0)$ and $M\ol{a}^T = (f,0)^T$  give us 120 affine linear equations 
in the 48 variables $A_{ij}, B_{ij}, C_{ij}$. 
For example, the first entry of the vector $aM$ is 
\begin{small}
\begin{align*}
&(A_{11} + i A_{21} - 3 i A_{31} - A_{41}) x^4 
+ (4 A_{31} - 2 i A_{41} + B_{11} + i B_{21} - 3 i B_{31} -  B_{41} ) x^3 y \\
&+ (-2 A_{11} + 4 i A_{21} + 4 i A_{31} + 4 B_{31} -  2 i B_{41} ) x^2 y^2 
+ (-4 A_{31} - 2 B_{11} + 4 i B_{21} + 4 i B_{31} ) x y^3 \\
&-  4 B_{31} y^4 + (-4 A_{21} - 2 i A_{41} + C_{11} + i C_{21} - 3 i C_{31} -   C_{41} ) x^3 z \\
&+ (4 A_{41} - 4 B_{21} - 2 i B_{41} + 4 C_{31} -2 i C_{41} ) x^2 y z 
+ (-4 A_{21} + 4 B_{41} - 2 C_{11} + 4 i C_{21} 
+   4 i C_{31}) x y^2 z \\
&+ (-4 B_{21} - 4 C_{31} ) y^3 z + (-2 A_{11} - 4 C_{21} -   2 i C_{41}) x^2 z^2 
+ (4 A_{31} - 2 B_{11} + 4 C_{41}) x y z^2 \\
&+ (4 B_{31} -  4 C_{21} ) y^2 z^2 + (4 A_{21} - 2 C_{11}) x z^3 
+ (4 B_{21} + 4 C_{31}) y z^3 +   4 C_{21} z^4.
\end{align*}
\end{small}

Identifying this polynomial with $f$ gives us 15 affine linear equations. For example, 
from the monomial $x^4$, we see that $A_{11} + i A_{21} - 3 i A_{31} - A_{41}=1$.
Similarly, from each of the other entries of $aM-(f,0)$ and $M\ol{a}^T - (f,0)^T$
we get 15 affine linear equations in the $3\cdot 4^2$ variables  $A_{ij}, B_{ij}, C_{ij}$, 
for a total of $2\cdot4\cdot 15 = 120$. 
The unique solution to these 120 equations gives the Hermitian matrix representation 
\[ xA+yB+zC  \;\;=\;\; \frac{1}{8}
 \begin{pmatrix} 
14 x& 2 z& 2 i x - 2 y& 2 i (y - z)\\ 
2 z& x&   0& -i x + 2 y\\ -2 i x - 2 y& 0& x& i x - 2 z\\ 
-2 i (y - z)&   i x + 2 y& -i x - 2 z& 4 x\end{pmatrix},
 \]
whose determinant is $(1/256)\cdot f$. 
\end{Example}

Ideally we would like to carry out our algorithm symbolically, as in this example, but 
the required field extensions generally will be too large. Given a hyperbolic polynomial 
$f \in \Q[x,y,z]_d$, one can ask: What is the field extension necessary to carry out 
the above construction symbolically?

In fact, after computing the points $\sV(f,g)$ and splitting them as $S\cup \ol{S}$, all of the 
remaining steps in the algorithm only require linear algebra, and thus can be done with rational arithmetic. 
Unfortunately, for generic $g\in \R[x,y,z]_{d-1}$ the Galois group of the intersection 
$\sV(f,g)$ is the full symmetric group of size $d(d-1)$.   It seems very hard to 
say anything about the smallest field extension of the points $\sV(f, g)$ for 
special choices of the interlacing polynomial $g$.

\begin{Remark}
Since the primary computational task is to find the coordinates of the points in $\sV(f,g)$, it would be worth investigating whether there is a way to find the polynomials $a_{1j}$ without computing the points $S$ explicitly. 
The analogous $1$-dimensional problem is, given a positive real polynomial $f(t)$ of degree $2d$, to write $f(t)$ as $g(t) \overline{g}(t)$ for some degree $d$ polynomial $g\in \C[t]_{\leq d}$. This task is known as 
``polynomial spectral factorization'', see, for example, \cite{SpecFactorization}.
Both problems involve finding a basis of polynomials (in either $\C[t]_{\leq d}$ or $\C[x,y,z]_{d-1}/(f)$) that vanish
on a set of non-real points, given a real polynomial that vanishes on both the set and its complex conjugate. 
This higher dimensional analogue does not seem to have been considered, but strikes us as worthy of study.
\end{Remark}

\section{Numerical Implementation}\label{sec:Implementation}
Here we discuss our numerical implementation of this algorithm in {\tt Mathematica} 
which is available as supplementary material on the website 
\url{http://people.math.gatech.edu/~rsinn3/numdetrep.html}. 
Below we give computation times and an error measure for $100$ random tests for polynomials of degrees $3$-$15$. 
Overall, this method results in very fast computations compared to other available methods, although the accuracy becomes poor 
for large $d$. The numerical accuracy of the approximation of the common zeros of the hyperbolic polynomial $f$ and the interlacing 
polynomial $g$ (step (A1)) strongly affects the final error in the numerical determinantal representation. The following steps (A4) and (A5) in our algorithm 
are all linear algebra problems that can be solved accurately.

One issue with numerical computations is that the affine linear equations 
in (A5) are overdetermined -- there are $d^3+3d^2+2d$ equations in $3d^2$ variables. 
With small numerical errors, these equations no longer have a solution. 
In our implementation of the algorithm, we solve this by the standard method of taking a least squares solution to the system. 

The step that takes the most significant amount of computation time by far is computing the points $\sV(f,g)$. The following steps in our algorithm
\begin{itemize}
\item Computing the basis $(a_{11}, \hdots, a_{1d})$. 
\item Translating the polynomial equations (A5) into a system of linear equations. 
\item Solving the resulting least squares problem. 
\end{itemize}
take much less time in our implementation.

Overall this method finds (approximate) determinantal representations surprisingly fast. 
To test our code, we generated hyperbolic polynomials of degree $d$ by taking the determinant
of $x I+y(B+B^T)+z(C+C^T)$ where $B$ and $C$ are random $d\times d$ matrices whose entries are chosen randomly from a normal distribution with mean $1$ and standard deviation $0.5$. Any such determinant will be hyperbolic with respect to the point $[1:0:0]$. 
Averaging test times for 100 examples in each degree gave the computation times shown in Table~1.

\begin{center}
\begin{table}\label{table:compTimes}
\begin{center}
\begin{tabular}{c|c|c|c|c|c|c|c}
degree & {\bf 3} & {\bf 4} & {\bf 5} & {\bf 6} & {\bf 7} & {\bf 8} & {\bf 9} \\ \hline
time (sec) & 0.05 & 0.11 & 0.25 & 0.56 & 1.19 & 2.36 & 4.47 \\ \hline
time $\sV(f,g)$ & 0.04 & 0.09 & 0.20 & 0.45 & 0.96 & 1.89 & 3.58 \\ \hline
error & $5 \cdot 10^{-13}$ & $3 \cdot 10^{-12}$ & $2 \cdot 10^{-11}$ & $2 \cdot 10^{-10}$ & $4 \cdot 10^{-9}$ & $4 \cdot 10^{-8}$ & $1 \cdot 10^{-7}$ \\\hline
relative error & $2 \cdot 10^{-14}$ & $9 \cdot 10^{-14}$ & $2 \cdot 10^{-13}$ & $4 \cdot 10^{-13}$ & $2\cdot 10^{-12}$ & $1\cdot 10^{-11}$ & $9\cdot 10^{-12}$ \\ \hline
\end{tabular}

\bigskip

\begin{tabular}{c|c|c|c|c|c|c}
degree & {\bf 10} & {\bf 11} & {\bf 12} & {\bf 13} & {\bf 14} & {\bf 15} \\ \hline
time (sec) & 9.04 & 13.81 & 22.95 & 37.06 & 58.08 & 89.71 \\ \hline
time $\sV(f,g)$ & 6.49 & 11.20 & 18.72 & 30.37 & 47.84 & 74.31 \\ \hline
error & $2 \cdot 10^{-6}$ & $0.03$ & $6 \cdot 10^{-4}$ & $0.01$ & $5.89$ & $47827$ \\ \hline
relative error & $6\cdot 10^{-11}$ & $3 \cdot 10^{-7}$ & $5 \cdot 10^{-10}$ & $4 \cdot 10^{-9}$ & $5 \cdot 10^{-7}$ & $1 \cdot 10^{-4}$ \\ \hline
\end{tabular}
\bigskip
\end{center}
\caption{Average times and errors for computing determinantal representations.  }
\end{table}
\end{center}
\smallskip

Here ``error'' means the maximum over the absolute values of the
coefficients of the difference between the original polynomial $f$ and
the appropriately scaled determinant $c\cdot\det(M)$. 
We also found it useful to look at the
``relative error'', meaning the error divided by the
largest coefficient of $f$. This is a more reasonable measure for the accuracy in the coefficients of a hyperbolic polynomial because they tend to become very large in higher degree. Indeed, differences between the coefficients of a hyperbolic polynomial can grow exponentially with the degree, 
as in \cite[\S 4]{Speyer05}.

One additional source of numerical errors is the computation of the determinant of 
the output of our algorithm. Because of the size of this matrix, a symbolic computation of the determinant 
is infeasible and instead we compute it by interpolation on a random set of points on the unit circle. Then we use the interpolated polynomial to 
compute the errors in the coefficients. Numerical testing suggests that numerical error in the coefficients of the interpolated determinant are small.

For comparison, the only other known methods for computing definite determinantal 
representations are discussed in \cite{MR2962788}. Here, finding definite determinantal 
representations is already extremely time consuming for quintics ($d=5$) and 
practically infeasible for larger degrees ($d\geq 6$).
Thus the method described above provides a great improvement in computation ability. 

\begin{Remark}
As we use a random matrix of linear forms to generate our test polynomials, one might 
also think to compare the starting matrices to those in the output, 
rather then the coefficients of the determinants. The reason this cannot be done is 
that a given hyperbolic polynomial $f$ has an infinite collection of Hermitian determinantal representations 
$f = \det(M)$, even when considered up to the equivalence $M\sim U^* M U$ for 
$U\in GL(d,\C)$. There is no reason for the starting determinantal representation 
$xI +y(B+B^T) + z(C+C^T)$ and the output $xM_1+yM_2+zM_3$ to be equivalent, and so we
cannot meaningfully compare their entries. 
\end{Remark}

\begin{Remark}\label{Remark:NondegenerateIntersection}
  The algorithm presented here could also be used to construct a
  Hermitian determinantal representation of a non-hyperbolic
  polynomial $f$, under certain conditions.  If $f \in \R[x,y,z]_d$,
  then one would need a polynomial $g \in \R[x,y,z]_{d-1}$ such that
  the intersection of the two curves $\sV_\C(f) \cap \sV_\C(g)$ is
  transversal and contains no real points. Furthermore, the vanishing
  ideal of $S$ in step (A4) of our algorithm must have dimension $d$
  in degree $d-1$; this translates into a condition on the
  corresponding divisor on $\sV_\C(f)$ (see \cite[\S
  4.5]{Vinnikov12}). It is a non-trivial fact that this condition is
  always met in our original setup when $f$ is hyperbolic and $g$
  interlaces $f$. This follows from \cite[Thm.~5.4]{Vinnikov12} or the
  main result of \cite{MR3066450}; a fuller discussion of the
  non-degeneracy condition in terms of points on the Jacobian can be
  found in \cite{Vinnikov93}. In the non-hyperbolic case, the
  condition is at least met for a generic choice of $g\in
  \R[x,y,z]_{d-1}$. If $d$ is odd, then the reality condition can be
  obtained by picking $g$ to have no real points, for example by
  taking $g$ as a sum of squares. Otherwise it is not clear how to
  pick such a polynomial $g$.
\end{Remark}

{\small
\bibliography{lit}{}
\bibliographystyle{plain}
}

\end{document}